\documentclass{amsart}
 \newtheorem{thm}{Theorem}[section]
 \newtheorem{cor}[thm]{Corollary}
 \newtheorem{lem}[thm]{Lemma}
 \newtheorem{prop}[thm]{Proposition}
 
 \theoremstyle{definition}
 \newtheorem{defn}[thm]{Definition}
 \theoremstyle{remark}
 \newtheorem{rem}[thm]{Remark}
 
 \numberwithin{equation}{section}

\usepackage{amssymb}
\usepackage{amsmath}
\usepackage{amsthm}
\usepackage{enumerate}
\usepackage{verbatim}
\usepackage[active]{srcltx}
\usepackage{bm}
\usepackage[noadjust]{cite}

\newcommand{\R}{\boldsymbol{R}}

\newcommand{\C}{\boldsymbol{C}}
\newcommand{\trace}{\operatorname{trace}}

\newcommand{\herm}{\operatorname{Herm}}
\newcommand{\im}{\operatorname{Im}\!}

\newcommand{\re}{\operatorname{Re}\!}
\renewcommand{\phi}{\varphi}
\newcommand{\sgn}{\operatorname{sgn}}
\newcommand{\inner}[2]{\left\langle{#1},{#2}\right\rangle}
\newcommand{\spann}[1]{\left\langle{#1}\right\rangle}

\newcommand{\E}{\mathcal{E}}

\newcommand{\ep}{\varepsilon}

\newcommand{\pmt}[1]{{\begin{pmatrix} #1  \end{pmatrix}}}

\newcommand{\mycomment}[1]{}

\begin{document}

%
%
%
%
%
%
%
%
%

\title[Dualities of invariants on cuspidal edges]
 {Dualities of differential geometric invariants 
on cuspidal edges on flat fronts 
in the hyperbolic space and the de Sitter space}

\author[K. Saji]{Kentaro Saji}

\address[K. Saji]{%
Department of Mathematics, 
Graduate School of Science, 
Kobe University, 
Rokkodai 1-1, Nada, Kobe 
657-8501, Japan}
\email{sajiO\!\!\!amath.kobe-u.ac.jp}
\thanks{Partly supported by the JSPS KAKENHI Grant Number 18K03301 
and the JSPS KAKENHI Grant Number 17J02151.}
\author[K. Teramoto]{Keisuke Teramoto}
\address[K. Teramoto]{%
Department of Mathematics, 
Graduate School of Science, 
Kobe University, 
Rokkodai 1-1, Nada, Kobe 
657-8501, Japan}
\email{teramotoO\!\!\!amath.kobe-u.ac.jp}
\subjclass[2010]{Primary 53A55; Secondary 53A35, 57R45}

\keywords{Cuspidal edge, flat front, duality}


\begin{abstract}
{
We compute the differential geometric invariants
of cuspidal edges on flat surfaces in hyperbolic $3$-space
and in de Sitter space.
Several dualities of invariants are pointed out.
}
\end{abstract}
\maketitle
\section{Introduction}

In \cite{gmm}, 
one Weierstrass type representation formula for flat surfaces
in the hyperbolic $3$-space $H^3$ is obtained,
and two representations are given in \cite{kuy}.
In \cite{kuy}, a notion {\it flat fronts} for
flat surfaces which admit certain singularities
is defined and the Weierstrass type representation formula
is extended to the representation formula for flat fronts.
This formula represents flat fronts via 
two holomorphic functions (called the {\it Weierstrass data}\/).
In \cite{KRSUY}, 
criteria for cuspidal edges and swallowtails 
are given in terms of the Weierstrass data,
and it is shown that these singularities are 
generic singularities of 
flat fronts in $H^3$. 
For a flat fronts in $H^3$, it is known that 
its unit normal vector taking the value
in the de Sitter $3$-space $S^3_1$ 
(called the de Sitter Gauss map\/) 
is also a flat front in $S^3_1$,
and it is constructed from the same data. 
It is known that the sets of singular points of
these two flat fronts coincide.
{Flat fronts are special cases of linear Weingarten
fronts.
In \cite{kokuumelw}, a global Weierstrass type 
representation of linear Weingarten fronts 
is given, and 
global properties of them 
are studied and correspondence
between dual surfaces of them are given 
(see \cite{kokuumelw} for detail, see also \cite{gmmlw,kokubulw}).}

One can regard that these two flat surfaces
are in a dual relation.
It is introduced in \cite{izdual},
that a formulation considering
this duality as a double Legendrian fibration 
in contact geometry.
See \cite{istaka,ismandala} for studies of inear Weingarten surfaces
from this viewpoint.
On the other hand, 
a {\it front\/} is a surface in a $3$-space
with well-defined unit normal vector
even on the set of singular points.
Since there is a unit normal vector, 
fronts can be studied from the differential geometric viewpoint. 
Many geometric invariants of cuspidal edges and other 
singular points
are introduced and geometric properties of them
are investigated.
(see \cite{
ribaucour,fukuicesw,hhnsuy,hnuy,
istce,martimila,invce,
MSUY,OT,front,terapara,teraprin,terafocal}
for example). 
Since flat fronts are determined by the Weierstrass data,
it is natural to consider a relation between
the data and the invariants of singular points.
This relation can be regarded 
as geometric meanings of the data
on singular points.

In this paper, 
we give explicit expressions for invariants of cuspidal edges in terms of 
the Weierstrass data $(\alpha,\beta)$ (Theorem \ref{prop:curvatures}\/).
We can observe several dualities of
geometric invariants and singularities 
of cuspidal edges on flat fronts in $H^3$ and $S^3_1$.
Using such expressions, we characterize the condition that 
the set of singular points which consists of cuspidal edges 
is a line of curvature (Proposition \ref{prop:curvline}\/). 
Furthermore, we show a relation between lines of curvature 
and cone-like singular points (Corollary \ref{cor:conelike}\/). 
For geometric meanings of cone-like singularities of flat fronts in $H^3$, 
see \cite{gmisol,KRSUY,kuy}.

\section{Preliminaries}
\subsection{Fronts in $H^3$ and $S^3_1$}
Let $\R^4_1$ be the Lorentz-Minkowski $4$-space 
with the inner product $\inner{~}{~}=({-}{+}{+}{+})$.
Let
\begin{align*}
H^3 &=\{ x=(x_0,x_1,x_2,x_3) \in \R^4 _1 \vert 
\inner{x}{x} = -1 ,x_0 >0  \},\\
S^3_1 &=\{ x \in \R^4 _1 \vert \inner{x}{x} = 1\}
\end{align*}
be the hyperbolic and the de Sitter $3$-spaces.
Although dual relations of surfaces in 
these two pseudo-spheres has been known
(see \cite[(2.3)]{kruycaus},\cite[(1.9)]{kruyasym} for example),
we use the formulation in \cite{izdual} to
working on Legendrian dualities.
Following \cite{izdual}, we set $\Delta_1\subset H^3\times S^3_1$ by
$$
\Delta_1
=\{(x,y)\in H^3\times S^3_1\,|\,
\inner{x}{y}=0\},
$$
and set
$
\pi_1:\Delta_1\to H^3$,
$
\pi_2:\Delta_1\to S^3_1
$
by
$$
\pi_1(x,y)=x,\quad
\pi_2(x,y)=y.
$$
Moreover, we set two $1$-forms
$$
\theta_1=-x_0\,dy_0+\sum_{i=1}^3x_i\,dy_i\Big|_{\Delta_1},\quad
\theta_2=-y_0\,dx_0+\sum_{i=1}^3y_i\,dx_i\Big|_{\Delta_1}
$$
Then $\theta_1^{-1}(0)$ and $\theta_2^{-1}(0)$ 
determine the same tangent plane field over $\Delta_1$
which is denoted by $K$.
It is well-known that the pair $(\Delta_1,K)$ is a contact
manifold and $\pi_1$ and $\pi_2$ are Legendrian fibrations
(\cite[Theorem 2.2]{izdual}).
This formulation is introduced for investigating
surfaces in the lightcones.
See \cite{izdual} for other dualities among pseudo-spheres.
It should be remarked that dualities of hypersurfaces 
in pseudo-spheres
in the Lorentz-Minkowski
space are also found independently (see \cite{egm,liujung,liuuy}).

Let $(N,ds^2)$ be a Riemannian or a semi-Riemannian $3$-manifold.
Let $U\subset \R^2$ be a domain.
A map $k:U\to N$ is a {\it frontal\/}
if there exists a map $L: U\to T_1N$ such that
$\pi\circ L=k$ and $ds^2(dk(X),L(p))=0$, where
$\pi:T_1N\to N$ is the unit tangent bundle.
The lift $L$ is called an {\it isotropic lift\/}
of $k$.
A frontal $k$ is a {\it front\/} if an isotropic lift
can be taken as an immersion.
Since the unit tangent bundle of $H^3$ can be
identified with $H^3\times S^3_1$, we can rewrite this 
setting by using $\Delta_1$.
A map $L=(f,g):U\to\Delta_1$ is {\it isotropic\/}
if $L^*\theta_1=0$.
A map $f:U\to H^3$ (respectively, $g:U\to S^3_1$)
is a frontal if there exists a map $g:U\to S^3_1$ 
(respectively, $f:U\to H^3$)
such that the pair $(f,g):U\to \Delta_1$ is isotropic.
A frontal $f:U\to H^3$ (respectively, $g:U\to S^3_1$)
is a front if the map $L=(f,g)$ can be taken as an
immersion.
If $(f,g):U\to\Delta_1$ is isotropic, 
then we say that $f$ and $g$ are $\Delta_1$-{\it dual\/} each other,
$g$ is a $\Delta_1$-{\it dual\/} of $f$, and
$f$ is a $\Delta_1$-{\it dual\/} of $g$.

\subsection{Matrix representation of $H^3$ and $S^3_1$}
Let $\herm(2)$ be the set of $2\times 2$ Hermitian matrices.
We take elements $e_0$, $e_1$, $e_2$, $e_3\in \herm(2)$ as
$$
e_0
=\pmt{1&0\\0&1},\quad
e_1
=\pmt{0&1\\1&0},\quad
e_2
=\pmt{0&i\\-i&0},\quad
e_3
=\pmt{1&0\\0&-1},
$$
where $i=\sqrt{-1}$.
Then we have an identification $\iota:\R^4\to \herm(2)$
$$
\iota(x)
=
\sum_{k=0}^3 x_k e_k =
\pmt{
x_0 +x_3  & x_1 + i x_2 \\
x_1 - i x_2 & x_0 - x_3 
},
$$
where $x=(x_0, x_1, x_2, x_3)$
with the metric
$$
\inner{X}{Y}=
-\dfrac{1}{2}\trace\Big(X\,e_2\,Y\,e_2\Big),
$$
$X,Y\in \herm(2)$.
In particular, 
$$
\inner{X}{X} =-  \det X \quad(X\in \herm(2)).
$$
By this identification,
$H^3$ and $S^3_1$ are rewritten as
\begin{align*}
H^3&=\{X\in\herm(2)\,|\,\det X=1, \trace X>0\}
=\{AA^*\,|\,A\in SL(2,\C)\},\\
S^3_1&=\{X\in\herm(2)\,|\,\det X=-1\}
=\{Ae_3A^*\,|\,A\in SL(2,\C)\}.
\end{align*}
Furthermore, the exterior product in
$T_pH^3$ and $T_pS^3_1$ are rewritten as
$$
X \times Y =\frac{i}{2}(Xp^{-1}Y-Y p^{-1}X)
$$
for $X,Y\in T_pH^3$, or $X,Y\in T_pS^3_1$.

\subsection{Singularities of fronts and their differential geometric invariants}
Let $(N,ds^2)$ be a Riemannian or a semi-Riemannian $3$-manifold.
Let $U\subset \R^2$ be a domain,
and $k:U\to N$ a front.
For $p\in U$, considering 
$L(p)=(p,\nu(p))\in T_p^1N$,
we consider the {\it signed area density function\/}
$$
\Omega(k_u(u,v),k_v(u,v),\nu(u,v))
$$
for a coordinate system $(u,v)$, and 
$(~)_u=\partial/\partial u$,
$(~)_v=\partial/\partial v$,
where $\Omega$ is the volume form.
A function $\lambda:U\to \R$ is called an
{\it identifier of singularity\/}
if it is a non-zero functional multiplication of 
$\Omega(k_u(u,v),k_v(u,v),\nu(u,v))$.
A singular point $p$ of $k$ is called {\it non-degenerate\/}
if $d\lambda(p)\ne0$, where $\lambda$ is an identifier of singularity.

Let $p\in U$ be a non-degenerate singular point 
of a front $k:U\to N$.
Then by the non-degeneracy of $p$, 
the set of singular points $\Sigma(k)$ is a regular curve near $p$,
and hence we take a parameterization $\gamma(t)$ $(\gamma(0)=p)$ of it.
We call $\gamma$ a {\it singular curve}.
We set $\hat\gamma=k\circ\gamma$.
One can show that
there exists a vector field $\eta$
on a sufficient small neighborhood of $p$
such that if $p\in \Sigma(k)$, then
$
\ker dk_p=\spann{\eta_p}.
$
We call $\eta$ the {\it null vector field}.
The restriction of $\eta$ on $\Sigma(k)$ can be parameterized by
the parameter $t$ of $\gamma$.
We denote by $\eta(t)$ the null vector field along $\gamma$.
\begin{defn}
Let $\R^3$ be the Euclidean $3$-space and 
$k:(\R^2,0)\to(\R^3,0)$ a $C^\infty$ map-germ.
The map-germ $k$ at $0$ is a {\it cuspidal edge\/}
(respectively, a {\it swallowtail\/})
if is is ${\mathcal A}$-equivalent to the map-germ
$(u,v)\mapsto(u,v^2,v^3)$ 
(respectively, $(u,v)\mapsto(u,4v^3+2uv,3v^4+uv^2)$)
at the origin.
Here, two map-germs $g_1,g_2:(\R^2,0)\to(\R^3,0)$
are {\em ${\mathcal A}$-equivalent\/}
if there exist diffeomorphism-germs
$\Xi_s:(\R^2,0)\to(\R^2,0)$ and
$\Xi_t:(\R^3,0)\to(\R^3,0)$
satisfying $g_2\circ \Xi_s=\Xi_t\circ g_1$.
\end{defn}
If $k:(\R^2,0)\to(\R^3,0)$ is a cuspidal edge or
a swallow, then it is a front and $0$ is a non-degenerate
singular point.
There are useful criteria for them.
Let $k:U\to \R^3$ be a front and $p\in U$ a non-degenerate
singular point.
Let $\gamma(t)$ be a parametrization $(\gamma(0)=p)$
of the singular curve, and 
$\eta(t)$ a null vector field.
We set $\delta(t)=\det(\gamma'(t),\eta(t))$.
For a non-degenerate singular point $p$ of $k$,
the map-germ $k$ at $p$ is cuspidal edge 
(respectively, swallowtail) if and only if
$\delta(0)\ne0$
(respectively, $\delta(0)=0, \delta'(0)=0$).
See \cite[Proposition 1.3]{KRSUY}.

Let $M$ be $H^3$ or $S^3_1$.
Let $f:U\to M$ be a front and let $f$ at $p$ 
be a cuspidal edge, and $g$ a $\Delta_1$-dual of $f$.
A pair of vector fields $(\xi,\eta)$ is called
{\it adapted\/}
if $\xi$ is tangent to the singular set and
$\eta$ is a null vector field of $f$.
By the criterion for cuspidal edge, $\eta\lambda\ne0$ holds.
Thus 
$\xi f$ and $\nabla_\eta \eta f$ are linearly independent,
in particular, $\xi f\ne0$ and $\xi f\times\nabla_\eta \eta f\ne0$
hold.
We define
\begin{align}
\label{eq:ksknktkikc}
\begin{aligned}
\kappa_s(t)=&
\sgn(\lambda_\eta) \frac{\Omega(\xi f,\,\nabla_\xi(\xi f),\,g)}
{|\xi f|^3}\Bigg|_{(u,v)=\gamma(t)},\quad
\kappa_n(t)=
 \frac{\inner{\nabla_\xi(\xi f)}{g}}{|\xi f|^2}\Bigg|_{(u,v)=\gamma(t)},\\
\kappa_t(t)=&
 \frac{\Omega(\xi f,\,\nabla_\eta(\eta f),\,\nabla_{\xi}\nabla_{\eta}(\eta f))}
{|\xi f\times \nabla_{\eta}(\eta f)|^2}\\
&\hspace{10mm}
-
\frac{\Omega(\xi f,\,\nabla_{\eta}(\eta f),\,\nabla_{\xi}(\xi f))
\inner{\xi f}{\nabla_{\eta}(\eta f)}}
{|\xi f|^2|\xi f\times \nabla_{\eta}(\eta f)|^2}
\Bigg|_{(u,v)=\gamma(t)},\\
\kappa_c (t)=&
\frac{|\xi f|^{3/2} \Omega(\xi f,\ \nabla_{\eta}(\eta f),\
\nabla_{\eta}\nabla_{\eta}(\eta f))}
{|\xi f\times \nabla_{\eta}(\eta f)|^{5/2}}
\Bigg|_{(u,v)=\gamma(t)},
\end{aligned}
\end{align}
where $\gamma(t)$ is a parametrization of $\Sigma(f)$,
$\lambda$ is an identifier of singularity, and
$\Omega$ is the volume form and 
under the identification $T_p\R^4_1=\R^4_1$,
it can be calculated by
$\Omega(X,Y,Z)=\inner{X\times Y}{Z}$
for $X,Y,Z\in T_{f(p)}H^3$ or 
for $X,Y,Z\in T_{f(p)}S^3_1$. 
Here, $\nabla$ is the metric connection
defined
for a vector $\zeta\in T_{f(p)}H^3$,
$$
\nabla_\zeta k
=
\zeta k+\inner{\zeta k}{f}f\in T_{f(p)}H^3,
$$
and
for a vector $\zeta\in T_{f(p)}S^3_1$,
$$
\nabla_\zeta k
=
\zeta k-\inner{\zeta k}{g}g\in T_{f(p)}S^3_1.
$$
The above functions $\kappa_s$, $\kappa_n$, $\kappa_t$ and $\kappa_c$
do not depend on the choices of the parameters of $\Sigma(f)$, $\Sigma(g)$
nor on the choice of $(\xi,\eta)$, and they are called
{\it singular curvature\/}, {\it limiting normal curvature\/},
{\it cusp-directional torsion\/} 
({\it cuspidal torsion\/})
and {\it cuspidal curvature\/} respectively.
See \cite{front} (for $\kappa_s$ and $\kappa_n$), 
\cite{invce} (for $\kappa_t$) and \cite{MSUY} (for $\kappa_c$) 
for detail.

\section{Flat fronts in $H^3$ and $S^3_1$}\label{sec:ffhs}
Let $U\subset\C$ be a domain, and 
$\alpha ,\beta:U\to\C\setminus\{0\}$ holomorphic functions.
We consider a 
solution $A:U \to SL(2,\C)$ 
of the differential equation
\begin{equation}\label{eq:aprime}
A'(z)
=A(z)D(z)\quad
\left(D(z)=
\pmt{
0  & \alpha(z) \\
\beta(z) & 0
}\right),\quad
{}'=\frac{d}{dz}.
\end{equation}
Then one can see that 
$\det A$ is a constant, 
we take a solution satisfying $\det A=1$.
We define maps
$f:U\to H^3$ and $g:U\to S^3_1$ by
\begin{equation}\label{eq:fandg}
f(z)=A(z)A^*(z),\quad
g(z)=A(z)e_3A^*(z).
\end{equation}
Then it is known that $f$ is (zero intrinsic curvature)
and $g$ is spacelike flat on their regular point sets
\cite[Proposition 2.5]{kuy}.
Furthermore, $f$ and $g$ are $\Delta_1$-dual each other.

By a calculation, we have
\begin{equation}\label{eq:fgbibun}
f' = ADA^*,\quad
f_{\overline{z}} = AD^*A^*,\quad
g' = ADe_3A^*,\quad
g_{\overline{z}} = Ae_3D^*A^*.
\end{equation}
Thus by
$$
f' \times f_{\overline{z}}
=\frac{i}{2}
\lambda
A
e_3A^*,\quad
g' \times g_{\overline{z}}
=-\frac{i}{2}
\lambda
AA^*,
$$
the singular sets $\Sigma(f)$ and $\Sigma(g)$ are coincide,
where $\lambda$ is defined by
\begin{equation}\label{eq:deflambda}
\lambda=\alpha \overline{\alpha} -\beta \overline{\beta}.
\end{equation}
We see that $\lambda$
is an identifier of singularity for each $f$ and $g$.
{It is also known that 
the set of singular points of flat front 
in the $3$-sphere $S^3$ and that of its dual coincide.
In \cite{kitaume},
a characterization of the
flat torus in $S^3$ is obtained.}

\subsection{Criteria for singularities of $f,g$}
We give conditions for cuspidal edges and swallowtails appearing 
on the flat fronts in $H^3$ and in $S^3_1$
by the Weierstrass data $(\alpha,\beta)$.
The case of the flat fronts in $H^3$, it is obtained in
\cite[Theorem 1.1]{KRSUY}, and 
the case of the flat fronts in $S^3_1$, the similar conditions 
are obtained in \cite[Proposition 6]{fnssy}.
We shall state here conditions for cuspidal edges and swallowtails on 
flat fronts in $S^3_1$ in terms of the Weierstrass data $(\alpha,\beta)$.

Let $f:U\to H^3$ and $g:U\to S^3_1$ be flat fronts 
constructed by the pair $(\alpha,\beta)$ of 
holomorphic functions as in \eqref{eq:fandg}.
Let $p$ be a non-degenerate singular point of $f$ or $g$,
and $\gamma(t)$ a parametrization of singular curve near $p$.
The tangent vector $\dot{\gamma}(t)=(d/dt)\gamma(t)$ to $\gamma$ 
can be expressed by 
\begin{equation}\label{eq:tangent}
\xi(t)=-i\lambda_{\overline{z}}(\gamma(t))
=
-i
(\alpha \overline{\alpha}_{\overline{z}} -
\beta\overline{\beta}_{\overline{z}})(\gamma(t))
\partial_z+
i(\alpha' \overline{\alpha} -
\beta'\overline{\beta})(\gamma(t))
\partial_{\overline{z}}
\end{equation}
(see \cite{KRSUY}) under the identification
$T_pU$ with $\R^2$ and $\C$ via the correspondence
$$\zeta=a+bi\in\C\leftrightarrow (a,b)\in\R^2\leftrightarrow 
a\partial_u+b\partial_v=
\zeta\partial_z+\overline{\zeta}\partial_{\overline{z}},$$
where $z=u+iv$. 
On the other hand, 
following \cite[p322]{KRSUY}, 
one can take null vector fields as follows:
\begin{lem}\label{lem:nullvect}
The vector field\/ $\eta_h$ gives a
null vector field of\/ $f$ and\/
$\eta_g$ gives a null vector field of\/ $g$, where
\begin{equation}\label{eq:null}
\eta_h=\dfrac{i}{\sqrt{\alpha\beta}}
=\dfrac{i}{\sqrt{\alpha\beta}}\partial_z
-\dfrac{i}{\sqrt{\overline{\alpha}\overline{\beta}}}
\partial_{\overline{z}},
\quad
\eta_d=\dfrac{1}{\sqrt{\alpha\beta}}
=\dfrac{1}{\sqrt{\alpha\beta}}\partial_z
+\dfrac{1}{\sqrt{\overline{\alpha}\overline{\beta}}}
\partial_{\overline{z}}.
\end{equation}
\end{lem}
\begin{proof}
See \cite[p322]{KRSUY} for $\eta_h$.
The directional derivative $\eta_dg$ of $g$ in the direction of $\eta_d$ can be calculated as 
\begin{equation*}
\eta_d g=A
\begin{pmatrix} 0 
& -\dfrac{\alpha}{\sqrt{\alpha\beta}}
+\dfrac{\overline{\beta}}{\sqrt{\overline{\alpha}\overline{\beta}}} \\
\dfrac{\beta}{\sqrt{\alpha\beta}}
-\dfrac{\overline{\alpha}}{\sqrt{\overline{\alpha}\overline{\beta}}} 
& 0 \end{pmatrix}
A^\ast.
\end{equation*}
Since $\alpha\overline{\alpha}-\beta\overline{\beta}=0$ on $\Sigma(g)$, 
it follows that $\eta_dg=0$ on $\Sigma(g)$.
\end{proof}
We set
\begin{align}
C_h&=
\re
\left(\dfrac{i\lambda'}{\sqrt{\alpha\beta}}\right)
=\dfrac{\re
\left(i\sqrt{\alpha\beta}\lambda_{\overline{z}}\right)}
{|\alpha|^2}\nonumber\\
&\hspace{15mm}
=-\im
\left(\dfrac{\lambda'}{\sqrt{\alpha\beta}}
\right)
=\im
\left(\dfrac{\lambda_{\overline{z}}}{\sqrt{\overline{\alpha\beta}}}
\right)
=\dfrac{\im\left(\sqrt{\alpha\beta}\lambda_{\overline{z}}\right)}
{|\alpha|^2},\label{eq:ch}\\
C_d&
=
\re
\left(\dfrac{\lambda'}{\sqrt{\alpha\beta}}\right)
=
\dfrac{\re
\left(\sqrt{\alpha\beta}\lambda_{\overline{z}}\right)}{|\alpha|^2}
\nonumber\\
&\hspace{15mm}
=\im
\left(\dfrac{i\lambda'}{\sqrt{\alpha\beta}}\right)
=\im
\left(\dfrac{i\lambda_{\overline{z}}}
{\sqrt{\overline{\alpha}\overline{\beta}}}
\right)
=\dfrac{\im\left(i\sqrt{\alpha\beta}\lambda_{\overline{z}}\right)}
{|\alpha|^2}.\label{eq:cd}
\end{align}
Then we have the following proposition.
\begin{prop}\label{prop:criteria}
{\rm (I)}{\rm (\cite[Theorem 1.1]{KRSUY})}
Let $f:U\to H^3$ be a flat front constructed 
by the data $(\alpha,\beta)$ as in \eqref{eq:fandg}. 
Let $p$ be a non-degenerate singular point of $f$. 
Then 
\begin{enumerate}
\item\label{crit:ce} $f$ at $p$ is a cuspidal edge if and only if 
$C_h\neq0$ at $p$.
\item\label{crit:sw} $f$ at $p$ is a swallowtail if and only if 
$C_h=0$ and 
\begin{equation}\label{eq:swcond}
\re\left(\dfrac{S(\alpha)-S(\beta)}{\alpha\beta}\right)\neq0
\end{equation}
at $p$, where  $S(\alpha)$ is the Schwarzian derivative of the 
primitive function of $\alpha$ with respect to $z:$
$$S(\alpha)=\left(\dfrac{\alpha'}{\alpha}\right)'-\dfrac{1}{2}\left(\dfrac{\alpha'}{\alpha}\right)^2.$$
\end{enumerate}
{\rm (II)}
Let $g:U\to S^3_1$ be a flat front constructed by the data $(\alpha,\beta)$ as in \eqref{eq:fandg}. 
Let $p$ be a non-degenerate singular point of $g$. 
Then 
\begin{enumerate}
\item\label{crit:ced} $g$ at $p$ is a cuspidal edge if and only if 
$C_d\neq0$ at $p$.
\item\label{crit:swd} $g$ at $p$ is a swallowtail if and only if 
$C_d=0$ and \eqref{eq:swcond} at $p$.
\end{enumerate}
\end{prop}
\begin{proof}
See \cite[Theorem 1.1]{KRSUY} for the proof of (I).
Since 
$$\{z\,|\,(\eta_h f,\eta_d f)=(0,0)\}
=\{z\,|\,(\eta_h g,\eta_d g)=(0,0)\}=\emptyset,$$
and $f$ and $g$ are dual each other, 
$f$ and $g$ are fronts.
By \cite[Proposition  1.3]{KRSUY},
it is enough to show that the condition (II), \ref{crit:ced} 
in the proposition is equivalent to $\delta_d(0)\neq0$,
where
$
\delta_d(t)=\det(\xi(t),\eta_d(t))
$.
By
\begin{align}\label{eq:delta1}
\begin{aligned}
\delta_d(t)&=-i\left(\dfrac{\lambda_{\overline{z}}}
{\sqrt{\overline{\alpha}\overline{\beta}}}+
\dfrac{\lambda'}{\sqrt{\alpha\beta}}\right)(\gamma(t))
=-2i\re\left(\dfrac{\lambda_{\overline{z}}}
{\sqrt{\overline{\alpha}\overline{\beta}}}\right)(\gamma(t))\\
&=\dfrac{-2i\re\big(\sqrt{\alpha\beta}
\lambda_{\overline{z}}\big)(\gamma(t))}{|\alpha(\gamma(t))|^2}
\end{aligned}
\end{align}
and the relation 
$\overline{\alpha}\overline{\beta}=|\alpha|^4/(\alpha\beta)$ holds 
along $\gamma$,
the first assertion holds.

Next we show the second assertion of (II). 
We note that 
$$
\lambda'=|\alpha|^2\left(\frac{\alpha'}{\alpha}
-\frac{\beta'}{\beta}\right)
$$ on $\gamma$. 
Thus by \eqref{eq:delta1}, $\delta_d$ is proportional to 
$$\tilde{\delta}_d
=\re\left(\frac{1}{\sqrt{\alpha\beta}}
\left(\frac{\alpha'}{\alpha}-\frac{\beta'}{\beta}\right)\right).$$ 
We assume that $\delta_d(0)=0$, namely, $\tilde{\delta}_d(0)=0$. 
Then since
\begin{align*}
&\dfrac{d}{dt}
\Bigg(\bigg(\frac{1}{\sqrt{\alpha\beta}}
\left(\frac{\alpha'}{\alpha}-\frac{\beta'}{\beta}\right)\bigg)(\gamma(t))
\Bigg)\\
&=\Bigg\{\left(
\left(\dfrac{\alpha'}{\alpha}\right)'
-\left(\dfrac{\beta'}{\beta}\right)'
\right)
\dfrac{1}{\sqrt{\alpha\beta}}
-\dfrac{1}{2}\left(\dfrac{\alpha'}{\alpha}
-\dfrac{\beta'}{\beta}\right)
\left(\dfrac{\alpha'}{\alpha}+\dfrac{\beta'}{\beta}\right)\dfrac{1}
{\sqrt{\alpha\beta}}\Bigg\}\dot{\gamma}\\
&=\dfrac{\dot{\gamma}}{\sqrt{\alpha\beta}}(S(\alpha)-S(\beta)).
\end{align*}
Thus we have 
$$
\dfrac{d}{dt}\Big(\tilde{\delta}_d(\gamma(t))\Big)
=\dfrac{1}{2}\left(\dfrac{\dot{\gamma}}{\sqrt{\alpha\beta}}
(S(\alpha)-S(\beta))+
\overline{\left(\dfrac{\dot{\gamma}}{\sqrt{\alpha\beta}}(S(\alpha)-S(\beta))\right)}\right).
$$
Here it holds that 
\begin{align*}
\dfrac{\dot{\gamma}}{\sqrt{\alpha\beta}}
&=\dfrac{-i|\alpha|^2}{\sqrt{\alpha\beta}}\overline{\left(\dfrac{\alpha'}{\alpha}-\dfrac{\beta'}{\beta}\right)}
=i\dfrac{|\alpha|^2\sqrt{\overline{\alpha}\overline{\beta}}}
{\alpha\beta}\left(\dfrac{\alpha'}{\alpha}-\dfrac{\beta'}{\beta}\right),\\ 
\overline{\left(\dfrac{\dot{\gamma}}{\sqrt{\alpha\beta}}\right)}
&=\dfrac{i|\alpha|^2}{\sqrt{\overline{\alpha}\overline{\beta}}}\left(\dfrac{\alpha'}{\alpha}-\dfrac{\beta'}{\beta}\right)
\end{align*}
at $p=\gamma(0)$ because \eqref{eq:tangent} and $\tilde{\delta}_d(0)=0$. 
Hence we have 
\begin{align*}
\left.\dfrac{d}{dt}\tilde{\delta}_d\right|_{t=0}
&=\dfrac{1}{2}\left(\dfrac{\dot{\gamma}}{\sqrt{\alpha\beta}}(S(\alpha)-S(\beta))+
\overline{\left(\dfrac{\dot{\gamma}}{\sqrt{\alpha\beta}}(S(\alpha)-S(\beta))\right)}\right)(p)\\
&=i|\alpha|^2\sqrt{\overline{\alpha}\overline{\beta}}\left(\dfrac{\alpha'}{\alpha}-\dfrac{\beta'}{\beta}\right)
\left(\dfrac{S(\alpha)-S(\beta)}{\alpha\beta}
+\overline{\left(\dfrac{S(\alpha)-S(\beta)}{\alpha\beta}\right)}\right)(p)\\
&=\dfrac{i|\alpha|^4}{\sqrt{{\alpha}{\beta}}}\left(\dfrac{\alpha'}{\alpha}-\dfrac{\beta'}{\beta}\right)
\re\left(\dfrac{S(\alpha)-S(\beta)}{\alpha\beta}\right)(p).
\end{align*}
We note that $(d/dt)\delta_d\Big|_{t=0}\neq0$ is equivalent to 
$(d/dt)\tilde{\delta}_d\Big|_{t=0}\neq0$
under the condition $\delta_d(0)=0$. 
Thus by \cite[Proposition 1.3]{KRSUY} again, we have the conclusion.
\end{proof}
\begin{rem}
Assuming that $u(z)$ 
is a solution of the differential equation
\begin{equation}\label{eq:slform}
u''(z)-\alpha(z)u(z)=0,
\end{equation}
satisfying $\det A=1$, where $A=(u(z),u'(z))$.
then the matrix $A$
satisfies \eqref{eq:aprime}.
This equation \eqref{eq:slform} is called the {\it SL-form of
the hypergeometric equation}.
The flat fronts $AA^*$ and $Ae_3A^*$ 
in $H^3$ and $S^3_1$ constructed from the data $(\alpha,1)$
can be regarded generalizations of the Schwarz map 
of the hypergeometric equation.
They are called {\it hyperbolic} and 
{\it de Sitter Schwarz map\/} (\cite{syy1}, see also \cite{fnssy,syy2}).
Thus setting $\beta=1$, then all invariants given 
in Section \ref{sec:ffhs} can be regarded as invariants 
of these Schwarz map.
\end{rem}

\subsection{Geometric invariants of cuspidal edges}
For flat fronts $f:U\to H^3$ and $g:U\to S^3_1$ given by \eqref{eq:fandg}.
Let $p\in U$ be a singular point, and let $f$ at $p$ 
be a cuspidal edge.
Let $\gamma(t)$ be a parametrization of $\Sigma(f)$ or $\Sigma(g)$.
Let $\kappa_s^h$, $\kappa_t^h$ and $\kappa_c^h$ 
(respectively, $\kappa_s^d$, $\kappa_t^d$ and $\kappa_c^d$)
be the singular curvature, cuspidal torsion and the cuspidal curvature
of $f$ on $\Sigma(f)$ (respectively, of $g$ on $\Sigma(g)$).
Then we have the following theorem.
\begin{thm}\label{prop:curvatures}
It holds that
\begin{align}
\kappa_s^h(t)
=&
-\dfrac{|\lambda'|^2}
{4|\alpha|^2\left|\im\left(\sqrt{\alpha\beta}
\lambda_{\overline{z}}\right)\right|}\Bigg|_{(u,v)=\gamma(t)}
&=&
-\dfrac{|\lambda'|^2}
{4|\alpha|^4C_h}\Bigg|_{(u,v)=\gamma(t)}
\label{eq:kappash}\\
\kappa_t^h(t)
=&
\dfrac{
\re\big(\sqrt{\alpha\beta}\lambda_{\overline{z}}\big)}
{\im\big(\sqrt{\alpha\beta}\lambda_{{\overline{z}}}\big)}
\Bigg|_{(u,v)=\gamma(t)}
&=&
\dfrac{C_d}{C_h}
\Bigg|_{(u,v)=\gamma(t)}
\label{eq:kappath}\\
\kappa_c^h(t)
=&
\dfrac{4|\alpha|^2
\im\big(\sqrt{\alpha\beta}\lambda_{{\overline{z}}}\big)}
{|\im\big(\sqrt{\alpha\beta}\lambda_{{\overline{z}}}\big)|^{3/2}}
\Bigg|_{(u,v)=\gamma(t)}
&=&
\dfrac{4|\alpha|C_h}
{|C_h|^{3/2}}
\Bigg|_{(u,v)=\gamma(t)}
\label{eq:kappach}\end{align}
and
\begin{align}
\kappa_s^d(t)
=&
-\dfrac{|\lambda'|^2}
{4|\alpha|^2\left|\re\left(\sqrt{\alpha\beta}\lambda_{\overline{z}}\right)\right|}\Bigg|_{(u,v)=\gamma(t)}
&=&
-\dfrac{|\lambda'|^2}{4|\alpha|^4C_d}\Bigg|_{(u,v)=\gamma(t)}
\label{eq:kappasd}\\
\kappa_t^d(t)
=&-
\dfrac{
\im\big(\sqrt{\alpha\beta}\lambda_{{\overline{z}}}\big)}
{\re\big(\sqrt{\alpha\beta}\lambda_{{\overline{z}}}\big)}
\Bigg|_{(u,v)=\gamma(t)}
&=&
-\dfrac{C_h}{C_d}\Bigg|_{(u,v)=\gamma(t)}
\label{eq:kappatd}\\
\kappa_c^d(t)
=&
-\dfrac{4|\alpha|^2
\re\big(\sqrt{\alpha\beta}\lambda_{{\overline{z}}}\big)}
{|\re\big(\sqrt{\alpha\beta}\lambda_{{\overline{z}}}\big)|^{3/2}}
\Bigg|_{(u,v)=\gamma(t)}
&=&
-\dfrac{4|\alpha|C_d}{|C_d|^{3/2}}
\Bigg|_{(u,v)=\gamma(t)}
\label{eq:kappacd}.
\end{align}
In particular, both $\kappa_s^h$ and $\kappa_s^d$ are strictly negative if $f$ and $g$ have only cuspidal edges.
\end{thm}
We can observe dual relations 
between $\kappa_t^h$ and $\kappa_t^d$ and
between $\kappa_t^h$ and $C_d$ (respectively, $\kappa_t^d$ and $C_h$).
We will see this duality in Section \ref{sec:relation}.
We remark that since the Gaussian curvature is
bounded, $\kappa_\nu$ vanishes identically
for $f$ and $g$.
\begin{proof}
We first calculate the invariant
$\kappa_s$ for $f$ and $g$.
By a direct calculation, we have
\begin{align*}
\begin{aligned}
\xi f&=
iA\pmt{0&-\lambda_{\overline{z}}\alpha+\lambda'\overline{\beta}\\
-\lambda_{\overline{z}}\beta+\lambda'\overline{\alpha}&0}A^*,\\ 
\xi g&=
iA\pmt{0&\lambda_{\overline{z}}\alpha+\lambda'\overline{\beta}\\
-(\lambda_{\overline{z}}\beta+\lambda'\overline{\alpha})&0}A^*,
\end{aligned}
\end{align*}
and on the singular set, 
\begin{align}
\begin{aligned}
\label{eq:xifxigsing}
\xi f=
i\dfrac{-\lambda_{\overline{z}}\alpha
+\lambda'\overline{\beta}}{\alpha}
ADA^*=
2\overline{\beta}\sqrt{\dfrac{\beta}{\alpha}}C_hADA^*,\\
\xi g=
i\dfrac{\lambda_{\overline{z}}\beta
 +\lambda'\overline{\alpha}}{\beta}
ADe_3A^*
=
2i\sqrt{\overline{\alpha}\overline{\beta}}C_dADe_3A^*.
\end{aligned}
\end{align}
Thus on the singular set, we have
\begin{align*}
\nabla_{\xi}\xi f
&\equiv
-2i\overline{\beta}\sqrt{\dfrac{\beta}{\alpha}}C_h
A\tilde D_hA^*\mod \langle f\rangle_{\E(U)},\\
\nabla_{\xi}\xi g
&\equiv
-2i\overline{\alpha}\sqrt{\dfrac{\alpha}{\beta}}C_d
A\tilde D_dA^*\mod \langle g\rangle_{\E(U)},
\end{align*}
where
$$
\tilde D_h=
-\lambda_{\overline{z}}D^2-\lambda_{\overline{z}}D'+\lambda'DD^*,\quad
\tilde D_d=
-\lambda_{\overline{z}}D^2e_3-\lambda_{\overline{z}}D'e_3+\lambda'De_3D^*.
$$
Here $\E(U)=\{h:U\to\R\}$
is the ring consists of function-germs on $U$, and
$\langle k_1,\ldots,k_r\rangle_{\E(U)}
=\{a_1k_1+\cdots+a_rk_r\,|\,a_1,\ldots,a_r\in {\E(U)}\}$.
The formula
$f_1\equiv f_2\mod A$
stands for $f_1-f_2\in A$.
Hence, we have
$$
\xi f\times \nabla_{\xi}\xi f
\equiv
4\overline{\beta}^2\dfrac{\alpha}{\beta}C_h^2
\lambda_{\overline{z}}(\alpha\beta'-\alpha'\beta)g
\equiv
-4\overline{\beta}^2\dfrac{\alpha^2}{\beta\overline{\beta}}C_h^2
|\lambda'|^2g
\mod \langle \xi f\times f\rangle_{\E(U)}
$$
and
$$
\xi g\times \nabla_{\xi}\xi g
\equiv
-4\overline{\alpha}\overline{\beta}C_d^2
\lambda_{\overline{z}}(\alpha\beta'-\alpha'\beta)f
\equiv
4\alpha\overline{\alpha}C_d^2
|\lambda'|^2f
\mod\langle \xi g\times g\rangle_{\E(U)}
$$
on the singular set, where we see
$
\lambda_{\overline{z}}(\alpha\beta'-\alpha'\beta)
=
-\lambda_{\overline{z}}
(\alpha/\overline{\beta})
(\alpha\alpha'-\beta\beta')
$ on the singular set.
On the other hand, the signed area density function
for $f$ and $g$ are 
$\Lambda_h=\Omega(\xi f,\eta_h f,g)$ and
$\Lambda_d=\Omega(\xi g,\eta_d g,f)$, respectively.
Then we see that 
$\eta_h\Lambda_h=-|\lambda_{\eta_h}|^2$, and
$\eta_d\Lambda_d=|\lambda_{\eta_d}|^2$.
Thus we have \eqref{eq:kappash} and \eqref{eq:kappasd}.

Now we proceed to calculate the cuspidal torsion and
the cuspidal curvature.
To show this, we prove several formulas of differentials
of $f$ and $g$ needed later.
\begin{lem}\label{lem:diff100}
On the singular set, it holds that
\begin{equation}\label{eq:onsingzero}
f'\times f_{\overline{z}}=f'\times f'_{\overline{z}}
=f_{\overline{z}}\times f'_{\overline{z}}=
g'\times g_{\overline{z}}=g'\times g'_{\overline{z}}
=g_{\overline{z}}\times g'_{\overline{z}}=0
\end{equation}
and
\begin{align}
\begin{aligned}\label{eq:fzfzz}
f'\times f''&=\dfrac{i}{2}(\alpha\beta'-\alpha'\beta)g,\quad 
f'\times f_{\overline{z}\overline{z}}
=\dfrac{i}{2}\lambda_{\overline{z}}g,\\
f_{\overline{z}}\times f''&=-\dfrac{i}{2}\lambda'g,\quad
f_{\overline{z}}\times f_{\overline{z}\overline{z}}
=\dfrac{i}{2}\overline{(\alpha'\beta-\alpha\beta')}g,
\end{aligned}\\
\begin{aligned}\label{eq:gzgzz}
g'\times g''&=\dfrac{i}{2}(\alpha\beta'-\alpha'\beta)f,\quad 
g'\times g_{\overline{z}\overline{z}}
=-\dfrac{i}{2}\lambda_{\overline{z}}f,\\
g_{\overline{z}}\times g''&=\dfrac{i}{2}\lambda'f,\quad
g_{\overline{z}}\times g_{\overline{z}\overline{z}}
=\dfrac{i}{2}\overline{(\alpha'\beta-\alpha\beta')}f.
\end{aligned}
\end{align}
In particular, all vectors in \eqref{eq:fzfzz} are
parallel to $g$, and that in \eqref{eq:gzgzz} are
parallel to $f$.
\end{lem}
\begin{proof}
By a direct calculation, we have
\begin{equation}\label{eq:fbibun}
f''
=
A\pmt{\alpha\beta & \alpha' \\ \beta' & \alpha\beta} A^*,\ 
f'_{\overline{z}}
=
A\pmt{|\alpha|^2 & 0 \\ 0 & |\beta|^2}A^*,\ 
f_{\overline{z}\overline{z}}
=
A\pmt{\overline{\alpha}\overline{\beta} & \overline{\beta'} \\ 
\overline{\alpha'} & \overline{\alpha}\overline{\beta}} A^*
\end{equation}
and
\begin{equation}\label{eq:gbibun}
g''
=
A\pmt{\alpha\beta & -\alpha' \\ \beta' & -\alpha\beta}A^*,\,
g'_{\overline{z}}
=
A\pmt{-|\alpha|^2 & 0 \\ 0 & |\beta|^2}A^*,\,
g_{\overline{z}\overline{z}}
=
A\pmt{\overline{\alpha}\overline{\beta} & \overline{\beta'} \\
-\overline{\alpha'} & -\overline{\alpha}\overline{\beta}}A^*.\\
\end{equation}
This shows the assertion.
\end{proof}
\begin{lem}\label{lem:cross1}
On the singular set, it holds that
\begin{align*}
\xi f\times \nabla_{\eta_h}\eta_hf
&=\dfrac{2}{|\alpha|^4}\left(\im\left(\sqrt{\alpha\beta}\lambda_{\overline{z}}\right)\right)^2g,\\
\xi f\times \nabla_{\eta_d}\eta_dg
&=\dfrac{-2}{|\alpha|^4}\left(\re\left(\sqrt{\alpha\beta}\lambda_{\overline{z}}\right)\right)^2f.
\end{align*}
\end{lem}
\begin{proof}
By definition, we have 
$$\nabla_{\eta_h}\eta_hf=\eta_h\eta_hf+\inner{\eta_h\eta_hf}{f}f,\quad
\nabla_{\eta_d}\eta_dg=\eta_d\eta_dg-\inner{\eta_d\eta_dg}{g}g.$$
Since $\inner{\eta_hf}{f}=\inner{\eta_dg}{g}=0$, it holds that 
$\inner{\eta_h\eta_hf}{f}+\inner{\eta_hf}{\eta_hf}=0$ and 
$\inner{\eta_d\eta_dg}{g}+\inner{\eta_dg}{\eta_dg}=0$. 
Moreover, on the singular set, 
$\eta_hf=\eta_dg=0$, 
and hence $\inner{\eta_h\eta_hf}{f}=\inner{\eta_d\eta_dg}{g}=0$ on $\Sigma(f)=\Sigma(g)$. 
In particular, it holds that $\nabla_{\eta_h}\eta_hf=\eta_h\eta_hf$ and 
$\nabla_{\eta_d}\eta_dg=\eta_d\eta_dg$, and 
we have $\xi f\times\nabla_{\eta_h}\eta_hf=\xi f\times\eta_h\eta_hf$ and 
$\xi g\times\nabla_{\eta_d}\eta_dg=\xi g\times\eta_d\eta_dg$ on the singular set. 

On the singular set, it follows that  
\begin{align}
\eta_h\eta_hf&=
\dfrac{-f''}{\alpha\beta}+\dfrac{2f'_{\overline{z}}}{|\alpha|^2}
-\dfrac{f_{\overline{z}\overline{z}}}{\overline{\alpha}\overline{\beta}}
+\dfrac{1}{2\alpha\beta}\left(\dfrac{\alpha'}{\alpha}
+\dfrac{\beta'}{\beta}\right)f'
+\dfrac{1}{2\overline{\alpha}\overline{\beta}}
\overline{\left(\dfrac{\alpha'}{\alpha}
+\dfrac{\beta'}{\beta}\right)}f_{\overline{z}},\label{eq:etaetaf}
\\
\eta_d\eta_dg&=
\dfrac{g''}{\alpha\beta}+\dfrac{2g'_{\overline{z}}}{|\alpha|^2}
+\dfrac{g_{\overline{z}\overline{z}}}{\overline{\alpha}\overline{\beta}}
-\dfrac{1}{2\alpha\beta}\left(\dfrac{\alpha'}{\alpha}+\dfrac{\beta'}{\beta}
\right)g'
-\dfrac{1}{2\overline{\alpha}\overline{\beta}}
\overline{\left(\dfrac{\alpha'}{\alpha}
+\dfrac{\beta'}{\beta}\right)}g_{\overline{z}}.\label{eq:etaetag}
\end{align}
By \eqref{eq:onsingzero}, \eqref{eq:fzfzz}, \eqref{eq:gzgzz}, \eqref{eq:etaetaf}, \eqref{eq:etaetag} 
and the relation $|\alpha|^2(\alpha'\beta-\alpha\beta')=\alpha\beta\lambda'$ 
which holds on $\Sigma(f)=\Sigma(g)$, we have the assertion.
\end{proof}
By Lemma \ref{lem:cross1}, for any $X\in T_{f(p)}H^3$ and 
$Y\in T_{g(p)}S^3_1$, we have 
\begin{align}\label{eq:omega2}
\begin{aligned}
\Omega(\xi f,\nabla_{\eta_h}\eta_hf,X)
&=
\dfrac{2}{|\alpha|^4}\left(\im\left(\sqrt{\alpha\beta}
\lambda_{\overline{z}}\right)\right)^2\inner{g}{X},\\
\Omega(\xi g,\nabla_{\eta_d}\eta_dg,Y)
&=
\dfrac{-2}{|\alpha|^4}\left(\re\left(\sqrt{\alpha\beta}
\lambda_{\overline{z}}\right)\right)^2\inner{f}{Y}.
\end{aligned}
\end{align}

\begin{lem}\label{lem:etaetaxi}
On the singular set, 
$\inner{\nabla_{\eta_h}\eta_hf}{\xi f}=
\inner{\nabla_{\eta_d}\eta_dg}{\xi g}=0$ holds.
\end{lem} 
\begin{proof}
By the above arguments, it follows that 
$\inner{\nabla_{\eta_h}\eta_hf}{\xi f}=\inner{\eta_h\eta_hf}{\xi f}$ and 
$\inner{\nabla_{\eta_d}\eta_dg}{\xi g}=\inner{\eta_d\eta_dg}{\xi g}$ 
on $\Sigma(f)=\Sigma(g)$. 
By \eqref{eq:fbibun} and \eqref{eq:gbibun}, 
\begin{align*}
\inner{f'}{f'}&=\alpha\beta,\quad
\inner{f'}{f_{\overline{z}}}=|\alpha|^2,\quad
\inner{f_{\overline{z}}}{f_{\overline{z}}}=\overline{\alpha}\overline{\beta},\\
\inner{f'}{f''}&=\dfrac{1}{2}(\alpha'\beta+\alpha\beta'),\quad
\inner{f'}{f'_{\overline{z}}}=0,\quad
\inner{f'}{f_{\overline{z}\overline{z}}}=\dfrac{1}{2}(\alpha\overline{\alpha'}+\beta\overline{\beta'}),\\
\inner{f_{\overline{z}}}{f''}&=\dfrac{1}{2}(\alpha'\overline{\alpha}+\beta'\overline{\beta}),\quad
\inner{f_{\overline{z}}}{f'_{\overline{z}}}=0,\quad
\inner{f_{\overline{z}}}{f_{\overline{z}\overline{z}}}=\dfrac{1}{2}\overline{(\alpha'\beta+\alpha\beta')},\\
\inner{g'}{g'}&=-\alpha\beta,\quad 
\inner{g'}{g_{\overline{z}}}=|\alpha|^2,\quad
\inner{g_{\overline{z}}}{g_{\overline{z}}}=-\overline{\alpha}\overline{\beta}\\
\inner{g'}{g''}&=\dfrac{-1}{2}(\alpha'\beta+\alpha\beta'),\quad
\inner{g'}{g'_{\overline{z}}}=0,\quad
\inner{g'}{g_{\overline{z}\overline{z}}}=\dfrac{1}{2}(\alpha\overline{\alpha'}+\beta\overline{\beta'}),\\
\inner{g_{\overline{z}}}{g''}&=\dfrac{1}{2}(\alpha'\overline{\alpha}+\beta'\overline{\beta}),\quad
\inner{g_{\overline{z}}}{g'_{\overline{z}}}=0,\quad
\inner{g_{\overline{z}}}{g_{\overline{z}\overline{z}}}=\dfrac{-1}{2}\overline{(\alpha'\beta+\alpha\beta')}
\end{align*}
hold on the singular set. 
Moreover, we have 
$\alpha'\overline{\alpha}+\beta'\overline{\beta}
=|\alpha|^2\left({\alpha'}/{\alpha}+{\beta'}/{\beta}\right)$,
$\alpha'\beta+\alpha\beta'=\alpha\beta
\left({\alpha'}/{\alpha}+{\beta'}/{\beta}\right)$ 
on the singular set.
Thus we have $\inner{\eta_h\eta_hf}{\xi f}=\inner{\eta_d\eta_dg}{\xi g}=0$ on the singular set. 
\end{proof}
We turn to calculate $\kappa_t^h$ and $\kappa_t^d$.
By Lemma \ref{lem:etaetaxi}, 
the second terms of $\kappa_t^h$ and $\kappa_t^d$ 
in \eqref{eq:ksknktkikc} vanish.
Thus for the calculations of them,
only the first terms of $\kappa_t^h$ and $\kappa_t^d$ 
in \eqref{eq:ksknktkikc}
are needed.
Since 
$\inner{f''}{g}=\inner{f'_{\overline{z}}}{g}
=\inner{f_{\overline{z}\overline{z}}}{g}=0$ 
and $\inner{g''}{f}=\inner{g'_{\overline{z}}}{f}=
\inner{g_{\overline{z}\overline{z}}}{f}=0$ hold on $\Sigma(f)=\Sigma(g)$,
and by \eqref{eq:etaetaf} and \eqref{eq:etaetag}, 
it holds that
\begin{align*}
\nabla_{\xi}\nabla_{\eta_h}\eta_hf&\equiv-i\lambda_{\overline{z}}
\left(\dfrac{-1}{\alpha\beta}f'''+\dfrac{2}{|\alpha|^2}f''_{\overline{z}}
-\dfrac{1}{\overline{\alpha}\overline{\beta}}f'_{\overline{z}\overline{z}}\right)\\
&+i\lambda'\left(\dfrac{-1}{\alpha\beta}f''_{\overline{z}}+\dfrac{2}{|\alpha|^2}f'_{\overline{z}\overline{z}}
-\dfrac{1}{\overline{\alpha}\overline{\beta}}f_{\overline{z}\overline{z}\overline{z}}\right)\mod\mathcal{F},\\
\nabla_{\xi}\nabla_{\eta_d}\eta_dg&\equiv-i\lambda_{\overline{z}}
\left(\dfrac{1}{\alpha\beta}g'''+\dfrac{2}{|\alpha|^2}g''_{\overline{z}}
+\dfrac{1}{\overline{\alpha}\overline{\beta}}g'_{\overline{z}\overline{z}}\right)\\
&+i\lambda'\left(\dfrac{1}{\alpha\beta}g''_{\overline{z}}+\dfrac{2}{|\alpha|^2}g'_{\overline{z}\overline{z}}
+\dfrac{1}{\overline{\alpha}\overline{\beta}}g_{\overline{z}\overline{z}\overline{z}}\right)\mod\mathcal{G},
\end{align*} 
where $\mathcal{F}=
\langle{f,f'f_{\overline{z}},f'',f'_{\overline{z}},
f_{\overline{z}\overline{z}}}\rangle_{\E(U)}$, 
$\mathcal{G}=\langle{g,g'g_{\overline{z}},g'',
g'_{\overline{z}},g_{\overline{z}\overline{z}}}\rangle_{\E(U)}$. 
We show the following lemma:
\begin{lem}\label{lem:diff200}
On the singular set, it holds that
\begin{align}\label{eq:naiseki3}
\begin{aligned}
\inner{f'''}{g}&=\dfrac{-\alpha\beta}{2|\alpha|^2}\lambda',\quad
\inner{f''_{\overline{z}}}{g}=\dfrac{\lambda'}{2},\quad
\inner{f'_{\overline{z}\overline{z}}}{g}=\dfrac{\lambda_{\overline{z}}}{2},\quad
\inner{f_{\overline{z}\overline{z}\overline{z}}}{g}
=\dfrac{-\overline{\alpha}\overline{\beta}}{2|\alpha|^2}\lambda_{\overline{z}},\\
\inner{g'''}{f}&=\dfrac{\alpha\beta}{2|\alpha|^2}\lambda',\quad
\inner{g''_{\overline{z}}}{f}=\dfrac{\lambda'}{2},\quad
\inner{g'_{\overline{z}\overline{z}}}{f}=\dfrac{\lambda_{\overline{z}}}{2},\quad
\inner{g_{\overline{z}\overline{z}\overline{z}}}{f}
=\dfrac{\overline{\alpha}\overline{\beta}}{2|\alpha|^2}\lambda_{\overline{z}}.
\end{aligned}
\end{align}
\end{lem}
\begin{proof}
By a direct calculation, we have
\begin{align}\label{eq:fzzz}
\begin{aligned}
f'''
&=A
\begin{pmatrix} 
\alpha'\beta+2\alpha\beta' & \alpha''+\alpha^2\beta 
\\ \beta''+\alpha\beta^2 & 2\alpha'\beta+\alpha\beta'
\end{pmatrix} 
A^\ast,\ 
f''_{\overline{z}}
=A\begin{pmatrix} \alpha'\overline{\alpha} & \alpha|\beta|^2 
\\ \beta|\alpha|^2 & \beta'\overline{\beta}\end{pmatrix} A^\ast,\\
f'_{\overline{z}\overline{z}}
&=A\begin{pmatrix} 
\alpha\overline{\alpha'} & \overline{\beta}|\alpha|^2 \\
\overline{\alpha}|\beta|^2 & \beta\overline{\beta'} 
\end{pmatrix}
A^\ast,\ 
f_{\overline{z}\overline{z}\overline{z}}
=A
\begin{pmatrix} 
\overline{\alpha'}\overline{\beta}+2\overline{\alpha}\overline{\beta'} 
& \overline{\beta''}+\overline{\alpha}\overline{\beta}^2 \\
\overline{\alpha''}+\overline{\alpha}^2\overline{\beta} 
& 2\overline{\alpha'}\overline{\beta}+\overline{\alpha}\overline{\beta'} 
\end{pmatrix}
A^\ast,
\end{aligned}
\end{align}
and
\begin{align}\label{eq:gzzz}
\begin{aligned}
g'''
&=
A
\begin{pmatrix}
\alpha'\beta+2\alpha\beta' & -(\alpha''+\alpha^2\beta) \\
\beta''+\alpha\beta^2 & -(2\alpha'\beta+\alpha\beta')
\end{pmatrix}
A^\ast,\ 
g''_{\overline{z}}
=A
\begin{pmatrix} 
-\alpha'\overline{\alpha} & \alpha|\beta|^2 \\ 
-\beta|\alpha|^2 & \beta'\overline{\beta}
\end{pmatrix}
A^\ast,\\
g'_{\overline{z}\overline{z}}
&\!=\!
A
\begin{pmatrix}
-\alpha\overline{\alpha'} & -\overline{\beta}|\alpha|^2 \\
\overline{\alpha}|\beta|^2 & \beta\overline{\beta'}
\end{pmatrix}
A^\ast,\,
g_{\overline{z}\overline{z}\overline{z}}
\!=\!A
\begin{pmatrix} \overline{\alpha'}\overline{\beta}
+2\overline{\alpha}\overline{\beta'} & 
\overline{\beta''}+\overline{\alpha}\overline{\beta^2} \\
-(\overline{\alpha''}+\overline{\alpha^2}\overline{\beta}) & 
-(2\overline{\alpha'}\overline{\beta}+\overline{\alpha}\overline{\beta'})
\end{pmatrix}A^\ast.
\end{aligned}
\end{align}
This shows the assertion.
\end{proof}
Let us continue the calculations for $\kappa_t^h$ and $\kappa_t^d$.
By \eqref{eq:omega2} and \eqref{eq:naiseki3}, 
\begin{align*}
\Omega(\xi f,\nabla_{\eta_h}\eta_hf,\nabla_{\xi}\nabla_{\eta_h}\eta_hf)
&=\dfrac{i\left(\im\left(\sqrt{\alpha\beta}\lambda_{\overline{z}}\right)\right)^2}
{|\alpha|^8}(\alpha\beta(\lambda_{\overline{z}})^2-\overline{\alpha}\overline{\beta}(\lambda')^2)\\
&=\dfrac{4\re\left(\sqrt{\alpha\beta}\lambda_{\overline{z}}\right)
\left(\im\left(\sqrt{\alpha\beta}\lambda_{\overline{z}}\right)\right)^3}{|\alpha|^8},\\
\Omega(\xi g,\nabla_{\eta_d}\eta_dg,\nabla_{\xi}\nabla_{\eta_d}\eta_dg)&=
\dfrac{i\left(\re\left(\sqrt{\alpha\beta}\lambda_{\overline{z}}\right)\right)^2}{|\alpha|^8}
(\alpha\beta(\lambda_{\overline{z}})^2-\overline{\alpha}\overline{\beta}(\lambda')^2)\\
&=\dfrac{4\im\left(\sqrt{\alpha\beta}\lambda_{\overline{z}}\right)
\left(\re\left(\sqrt{\alpha\beta}\lambda_{\overline{z}}\right)\right)^3}{|\alpha|^8}
\end{align*}
hold on the singular set. 
Therefore by \eqref{eq:ksknktkikc}, 
the formulas
\eqref{eq:kappath} and \eqref{eq:kappatd}, are proven.
Finally, we consider the cuspidal curvatures $\kappa_c^h$ and $\kappa_c^d$. 
By \eqref{eq:xifxigsing},
\begin{equation*}
|\xi f|=2|\alpha|^2|C_h|=2\left|\im\left(\sqrt{\alpha\beta}\lambda_{\overline{z}}\right)\right|, \ \ 
|\xi g|=2|\alpha|^2|C_d|=2\left|\re\left(\sqrt{\alpha\beta}\lambda_{\overline{z}}\right)\right|
\end{equation*}
hold on the singular set.
Moreover, we have
\begin{align*}\allowdisplaybreaks[3]
\nabla_{\eta_h}\nabla_{\eta_h}\eta_hf&\equiv\dfrac{i}{\sqrt{\alpha\beta}}
\left(\dfrac{-1}{\alpha\beta}f'''+\dfrac{2}{|\alpha|^2}f''_{\overline{z}}
-\dfrac{1}{\overline{\alpha}\overline{\beta}}f'_{\overline{z}\overline{z}}\right)\\
&-\dfrac{i}{\sqrt{\overline{\alpha}\overline{\beta}}}
\left(\dfrac{-1}{\alpha\beta}f''_{\overline{z}}+\dfrac{2}{|\alpha|^2}f'_{\overline{z}\overline{z}}
-\dfrac{1}{\overline{\alpha}\overline{\beta}}f_{\overline{z}\overline{z}\overline{z}}\right)\mod\mathcal{F},\\
\nabla_{\eta_d}\nabla_{\eta_d}\eta_dg&\equiv\dfrac{1}{\sqrt{\alpha\beta}}
\left(\dfrac{1}{\alpha\beta}g'''+\dfrac{2}{|\alpha|^2}g''_{\overline{z}}
+\dfrac{1}{\overline{\alpha}\overline{\beta}}g'_{\overline{z}\overline{z}}\right)\\
&+\dfrac{1}{\sqrt{\overline{\alpha}\overline{\beta}}}
\left(\dfrac{1}{\alpha\beta}g''_{\overline{z}}+\dfrac{2}{|\alpha|^2}g'_{\overline{z}\overline{z}}
+\dfrac{1}{\overline{\alpha}\overline{\beta}}g_{\overline{z}\overline{z}\overline{z}}\right)\mod\mathcal{G}
\end{align*}
on the singular set. 
Thus we have 
\begin{align*}
\inner{g}{\nabla_{\eta_h}\nabla_{\eta_h}\eta_hf}
&=\dfrac{2i}{|\alpha|^2}\left(\dfrac{\lambda'}{\sqrt{\alpha\beta}}-\dfrac{\lambda_{\overline{z}}}
{\sqrt{\overline{\alpha}\overline{\beta}}}\right)
=\dfrac{4\im\left(\sqrt{\alpha\beta}\lambda_{\overline{z}}\right)}{|\alpha|^4},\\
\inner{f}{\nabla_{\eta_d}\nabla_{\eta_d}\eta_dg}
&=\dfrac{2}{|\alpha|^2}\left(\dfrac{\lambda'}{\sqrt{\alpha\beta}}
+\dfrac{\lambda_{\overline{z}}}{\sqrt{\overline{\alpha}\overline{\beta}}}\right)
=\dfrac{4\re\left(\sqrt{\alpha\beta}\lambda_{\overline{z}}\right)}{|\alpha|^4}
\end{align*}
on the singular set by \eqref{eq:naiseki3}. 
Hence we obtain \eqref{eq:kappach} and \eqref{eq:kappacd}. 
\end{proof}
We note that $\kappa_s^h<0$ follows from the general theory (\cite[Theorem 3.1]{front}) 
since the extrinsic Gaussian curvature of a flat front $f$ is $1>0$. 
However, $\kappa_s^d<0$ does not follow by the general theory since 
the extrinsic Gaussian curvature of a flat front $g$ is $-1<0$.

\subsection{Relationships between curvatures and singularities of the dual surfaces}\label{sec:relation}
For a non-degenerate singular point $p$ of $f$ (resp. $g$), 
it is known that $f$ at $p$ (resp. $g$ at $p$) is not a cuspidal edge 
if and only if $C_h(p)=0$ (resp. $C_d(p)=0$), where $C_h$ 
(resp. $C_d$) is as in \eqref{eq:ch} (resp. \eqref{eq:cd}) 
(cf. \cite{KRSUY}). 
By {Theorem} \ref{prop:curvatures}, it follows that 
a flat front $f$ (resp. $g$) in $H^3$ (resp. $S^3_1$) is not a cuspidal edge 
at a non-degenerate singular point $p$ 
if and only if $\kappa_t^d$ (resp. $\kappa_t^h$) vanishes at $p$.
Therefore we consider relation between the type of singularity 
which is not a cuspidal edge of $f$ or $g$ and  
behavior of $\kappa_t^d$ or $\kappa_t^h$. 

Let $U\subset \C$ be a simply-connected domain
and $\mathcal{O}(U)$ the set of holomorphic functions on $U$. 
Then, for $h\in\mathcal{O}(U)$, we can construct flat fronts 
$$f=f_h:U\to H^3,\quad g=g_h:U\to S^3_1$$
which are represented by a pair of holomorphic 
functions $(\alpha,\beta)=(e^h,1)$. 
Converse is also true, namely, for a flat front $f:U\to H^3$
without umbilical point,
then we can choose a suitable complex coordinate $z$ such that
$\alpha\,dz=e^h\,dz$ and $\beta\,dz=dz$.
(see \cite[p 323]{KRSUY}).

Let $f:U\to H^3$ and $g:U\to S^3_1$ be
flat fronts defined by \eqref{eq:fandg}
by the data $(\alpha,\beta)=(e^h,1)$.
Then the sets of singular points of $f$ and $g$ are 
$$\Sigma(f)=\Sigma(g)=\{z\in U\ |\ h(z)+\overline{h(z)}=0\}.$$
Moreover, $\lambda=h+\overline{h}$ is an identifier of singularity. 
By Proposition \ref{prop:criteria},
it holds that:
\begin{itemize}
\item $f$ at $p(=\gamma(0))$ is a cuspidal edge if and only if 
$\im\big(e^{-\frac{h}{2}}h'\big)\neq0$ at $p$,
\item $f$ at $p$ is a swallowtail if and only if 
$\im\big(e^{-\frac{h}{2}}h'\big)=0$ and 
$\re\big(e^{-h}(h''-\frac{1}{2}(h')^2)\big)\neq0$ at $p$.
\item $g$ is a cuspidal edge at $p$ if and only if 
$\re\big(e^{-\frac{h}{2}}h'\big)\neq0$ at $p$,
\item $g$ is a swallowtail at $p$ if and only if 
$\re\big(e^{-\frac{h}{2}}h'\big)=0$ and 
$\re\big(e^{-h}(h''-\frac{1}{2}(h')^2)\big)\neq0$ at $p$.
\end{itemize}
We have the following theorem.
\begin{thm}\label{thm:dualsing}
Let $f:U\to H^3$ and $g:U\to S^3_1$ be flat fronts and $p$ be a non-degenerate singular point of both $f$ and $g$. 
Let $\gamma(t)$ be a singular curve through $p=\gamma(0)$. 
Assume that $f$ at $p$ is a cuspidal edge and $\kappa_t^h(p)=0$ 
$($\/resp. $g$ at $p$ is a cuspidal edge and $\kappa_t^d(p)=0\/)$. 
Then $g$ $($\/resp. $f\/)$ is a swallowtail at $p$ 
if and only if $\left.\frac{d}{dt}{\kappa_t^h}\right|_{t=0}\neq0$ 
$($\/resp. $\left.\frac{d}{dt}{\kappa_t^d}\right|_{t=0}\neq0\/)$.
\end{thm}
\begin{proof}
Since this is a local situation, 
we may assume that both $f$ and $g$ are constructed by the data $(\alpha,\beta)=(e^h,1)$, for $h\in\mathcal{O}(U)$. 
Let $\gamma(t)$ be a parametrization of $\Sigma(f)=\Sigma(g)$.
By {Theorem} \ref{prop:curvatures}, 
$\kappa_c^h$, $\kappa_c^d$, $\kappa_t^h$ and $\kappa_t^d$ are given by
\begin{align}\label{eq:kappacteh}
\begin{aligned}
\kappa_c^h(t)&=\frac{4\im\big(e^{-\frac{h}{2}}h'\big)}
{|\im\big(e^{-\frac{h}{2}}h'\big)|^{3/2}}(\gamma(t)),\quad & 
\kappa_c^d(t)&=-\frac{4\re\big(e^{-\frac{h}{2}}h'\big)}
{|\re\big(e^{-\frac{h}{2}}h'\big)|^{3/2}}(\gamma(t)),\\
\kappa_t^h(t)&=\frac{\re\big(e^{-\frac{h}{2}}h'\big)}
{\im\big(e^{-\frac{h}{2}}h'\big)}(\gamma(t)),\quad & 
\kappa_t^d(t)&=\frac{\im\big(e^{-\frac{h}{2}}h'\big)}
{\re\big(e^{-\frac{h}{2}}h'\big)}(\gamma(t)).
\end{aligned}
\end{align}
First, we assume that $\kappa_t^d=0$ and $\kappa_c^d\neq0$ at $p$. 
This is equivalent to the condition that $\im\big(e^{-\frac{h}{2}}h'\big)=0$ 
and $\re\big(e^{-\frac{h}{2}}h'\big)\neq0$ at $p$. 
Thus $(d/dt)\kappa_t^d(0)\neq0$ if and only if 
$(d/dt)\im\big(e^{-\frac{h}{2}}h'\big)(0)\neq0$. 
By a direct calculation, we see that 
\begin{align*}
&\frac{d}{dt}\Big(\im\big(e^{-\frac{h}{2}}h'\big)(\gamma(t))\Big)\\
=&\frac{1}{2i}
\left(e^{-\frac{h}{2}}\left(h''-\frac{1}{2}(h')^2\right)\dot{\gamma}
+e^{-\frac{\overline{h}}{2}}\left(\overline{h''}
-\frac{1}{2}(\overline{h'})^2\right)\overline{\dot{\gamma}} \right)(t)\\
=&\frac{1}{2}
\left(-e^{-\frac{h}{2}}\left(h''-\frac{1}{2}(h')^2\right)\overline{h'}
+e^{-\frac{\overline{h}}{2}}\left(\overline{h''}
-\frac{1}{2}(\overline{h'})^2\right)h'\right)(t).
\end{align*}
Since we assume that $h+\overline{h}=0$ and 
$e^{-\frac{h}{2}}h'-e^{-\frac{\overline{h}}{2}}\overline{h'}=0$ at $p$, 
we have $\overline{h'}=e^{-h}h'$ and 
$h'=e^{-\overline{h}}\overline{h'}$. 
Hence it holds that 
\begin{align*}
\left.\frac{d}{dt}\Big(\im\big(e^{-\frac{h}{2}}h'\big)(\gamma(t))\Big)
\right|_{t=0}&=
\frac{1}{2}\left(-e^{-h}\left(h''-\frac{1}{2}(h')^2\right)
e^{-\frac{h}{2}}h'\right.\\
&\hspace{2cm}\left.+e^{-\overline{h}}\left(\overline{h''}
-\frac{1}{2}(\overline{h'})^2\right)
e^{-\frac{\overline{h}}{2}}\overline{h'}\right)(p)\\
&=-(e^{-\frac{h}{2}}h')(p)\re\left(\left(e^{-h}\left(h''-\frac{1}{2}(h')^2
\right)\right)(p)\right),
\end{align*}
where we used the relation 
$e^{-\frac{h}{2}}h'-e^{-\frac{\overline{h}}{2}}\overline{h'}=0$ at $p$ again
to have the second equality.
Thus we have the assertion for $f$. 

The case for $g$ can be proven by the same computation
using \eqref{eq:kappacteh}.

\end{proof}
{
In \cite[Proposition 3.8]{kokuumelw}, 
an equivalent statement of Theorem \ref{thm:dualsing}
is obtained from a different viewpoint.
In fact, $\kappa_t^h$ coincides with
$\Delta$ in \cite[p. 1910]{kokuumelw}
when $\ep=0$ (a linear Weingarten front is a flat front
if $\ep=0$) in their notation.
Thus Theorem \ref{thm:dualsing} clarifies
the geomteric meaning of $\Delta$ in \cite{kokuumelw}
for a flat front in $H^3$.
In \cite{kokuumelw}, it is shown that
the co-orientabilities and 
orientabilities of the original front and its dual
have a connection.
Furthermore, zig-zag numbers of them are studied.}


\section{Lines of curvature and cone-like singular points}
We consider the condition that the singular curve $\gamma$ of $f$ and $g$ 
are a line of curvature. 
Let $k:U\to N$ be a frontal into a 
Riemannian or semi-Riemannian $3$-manifold $N$. 
Let $I$ be an interval.
A curve $\gamma:I\to U$ is a {\it line of curvature\/} if 
${\rm I}(k\circ\gamma')$ is parallel to ${\rm I\!I}(k\circ\gamma')$, where 
${\rm I}$ and ${\rm I\!I}$ are the first and the second
fundamental matrices, and
they are regarded as linear maps $T_p\nu^\perp\to T_p\nu^\perp\subset T_pN$
and $k\circ\gamma'(t)\in T_p\nu^\perp$ $(p=\gamma(t))$.
\begin{prop}\label{prop:curvline}
Let $f:U\to H^3$ and $g:U\to S^3_1$ be flat fronts constructed 
by \eqref{eq:fandg}.
Let $p\in\Sigma(f)=\Sigma(g)$ be a cuspidal edge of $f$ $($\/resp. $g\/)$
and $\gamma$ a singular curve through $p=\gamma(0)$. 
Then $\gamma$ is a line of curvature of $f$ $($\/resp. $g\/)$ 
if and only if $\kappa_t^h$ $($\/resp. $\kappa_t^d\/)$ 
vanishes identically on $\gamma$.
\end{prop}
Although this is shown for the case of fronts in Euclidean $3$-space
(see \cite[Proposition 3.3]{teraprin}, \cite[p. 95]{istce}),
and the same proof works,
we give a proof here using the representation formula.
\begin{proof}
The curve $\gamma$ is a line of curvature of $f$ 
(\/resp. $g$\/) if and only if 
the following equation holds:
\begin{equation*}
\Omega(\xi f,g,\xi g)(\gamma(t))=0\quad (\/\text{resp.}\ \Omega(\xi g,f,\xi f)(\gamma(t))=0\/).
\end{equation*}
By a direct calculation, we see that 
\begin{align*}
\Omega(\xi f,g,\xi g)(\gamma(t))&=\inner{\xi f\times g}{\xi g}(\gamma(t))\\
&=4\re\big(\sqrt{\alpha\beta}\lambda_{\overline{z}}\big)(\gamma(t))
\im\big(\sqrt{\alpha\beta}\lambda_{\overline{z}}\big)(\gamma(t))\\
\big(\/\text{resp.}\ 
\Omega(\xi g,f,\xi f)(\gamma(t))&=\inner{\xi g\times f}{\xi f}(\gamma(t))\\
&=-4\re\big(\sqrt{\alpha\beta}\lambda_{\overline{z}}\big)(\gamma(t))
\im\big(\sqrt{\alpha\beta}\lambda_{\overline{z}}\big)(\gamma(t))\big).
\end{align*}
Thus we have the assertion by {Theorem} \ref{prop:curvatures} 
if $f$ (\/resp. $g$\/) has cuspidal edges along $\gamma$.
\end{proof}
We define another kind of singular point so called 
cone-like singularity (\cite[p. 305]{KRSUY}).
\begin{defn}
Let $h:U\to M$ be a frontal from a region into a $3$-dimensional
manifold. Let $p\in U$ be a non-degenerate singular point.
Then $p$ is a {\it cone-like singularity\/} if
there exists a neighborhood $V$ of $p$ such that
for any $q\in \Sigma(h)$,
a null vector field $\eta_q$ is tangent to $\Sigma(h)$.
\end{defn}
Assume that the singular curve $\gamma$ of $f$ (\/resp. $g$\/)
consists of cuspidal edges.
By Proposition \ref{prop:curvline}, if $\gamma$ is a 
line of curvature of $f$ (\/resp. $g$\/), 
then $\xi$ and $\eta_d$ (\/resp. $\xi$ and $\eta_h$) 
are parallel along $\gamma$. 
Thus we have the following corollary.
\begin{cor}\label{cor:conelike}
Let $f:U\to H^3$ and $g:U\to S^3_1$ be 
flat fronts and $p\in\Sigma(f)=\Sigma(g)$ a non-degenerate singular point. 
If the singular curve $\gamma$ passing through $p$ consists 
of cuspidal edges of $f$ 
$($\/resp. $g\/)$ and is a line of curvature 
of $f$ $($\/resp. $g\/)$, 
then $g$ $($\/resp. $f\/)$ has a cone-like singularity at $p$.
\end{cor}

\proof[Acknowledgments]
The authors would like to thank Masaaki Umehara
for valuable comments, and thank
Shun Iguchi and Mao Nomura for helping calculations. 


\end{document}